\documentclass[12pt]{article}
\usepackage{geometry}
\usepackage[utf8]{inputenc}
\usepackage{amsmath,amsthm,amssymb,color,graphicx,url,hyperref,diagbox,enumerate,cite,tikz,authblk}
\usepackage{appendix}
\usepackage{makecell}

\newtheorem{theorem}{Theorem}

\newtheorem{problem}[theorem]{Problem}

\newtheorem{lemma}[theorem]{Lemma}

\newtheorem{corollary}[theorem]{Corollary}

\newtheorem{definition}[theorem]{Definition}

\newcounter{encoding}
\newtheorem{encoding}[encoding]{Encoding}

\newcommand{\Z}{\mathbb{Z}}

\newcommand{\F}{\mathbb{F}}

\newcommand{\inv}{^{-1}}

\usepackage{graphicx} 

\title{Lower Bounds for Book Ramsey Numbers}

\author[1]{William J. Wesley \footnote{University of California, San Diego \tt{wjwesley@ucsd.edu}}}

\begin{document}

\maketitle

\begin{abstract}
    We prove new bounds for Ramsey numbers for book graphs $B_n$. In particular, we show that $R(B_{n-1},B_n) = 4n-1$ for an infinite family of $n$ using a block-circulant construction similar to Paley graphs. We obtain improved bounds for several other values of $R(B_r,B_s)$ using different block-circulant graphs from SAT and integer programming (IP) solvers. Finally, we enumerate the number of critical graphs for $R(B_r,B_s)$ for small $r$ and $s$ using SAT modulo symmetries (SMS). 
\end{abstract}
\section{Introduction}

The \emph{Ramsey number} $R(r,s)$ is the smallest $n$ such that every graph on $n$ vertices contains either a copy of the clique $K_r$ or coclique $\overline{K}_s$. Ramsey numbers are among the most widely studied numbers in combinatorics, yet pinning down their precise values is remarkably difficult. Only nine nontrivial values are known, namely $R(4,4), R(4,5)$, and $R(3,s)$ for $3 \le s \le 9$. 

A natural modification is to consider graphs other than $K_r$ and $K_s$, and instead compute $R(G,H)$, the smallest $n$ such that every graph on $n$ vertices contains either a copy of $G$ or $\overline{H}$. Note that $R(K_r,K_s)$ is simply $R(r,s)$. For suitable choices of $G$ and $H$, the problem of computing $R(G,H)$ becomes much simpler. These generalized Ramsey numbers are of great interest, both for their asymptotics and exact values. A comprehensive dynamic survey of their known values is maintained by Radziszowski \cite{RamseySurvey}.

In this work we restrict our attention to the \emph{book graphs} $B_n = K_2 + \overline{K_n}$, where + denotes the graph join. That is, $B_n$ consists of $n$ triangles that share a common edge (see Figure \ref{Figure_books}). The first systematic study of book Ramsey numbers was done in 1978 by Rousseau and Sheehan \cite{RousseauSheehanRamseyBooksOriginal}. The authors are rather modest and claim that ``the book $(B_n)$ does not have the same status in graph theory as does, for example, the complete graph $(K_n)$, the path $(P_n)$, cycle $(C_n)$, or wheel $(W_n)$". However, in recent years Ramsey numbers involving book graphs have attracted significant attention \cite{Conlon_BookRamsey,ConlonFoxWigderson_OffDiagonalBooks,LiuLiRamseyBooks,ChenLinRamseyBookUpperBounds}, and notably, book graphs are a key ingredient in the recent breakthrough proof that $R(s,s) \le (4 - \epsilon)^s$ \cite{ExponentialImprovementRamsey}. Rousseau and Sheehan proved several exact bounds for families of book Ramsey numbers in \cite{RousseauSheehanRamseyBooksOriginal}. In particular, they gave the following bounds for the diagonal and ``almost diagonal" cases. 
\begin{theorem}[Rousseau-Sheehan] \label{TheoremAlmostDiagonalBookUpperBound} For all $n$, 
\begin{enumerate}[(i)]
    \item $R(B_n,B_n) = 4n+2 \text{ if } q = 4n+1 \text{ is a prime power},$ \\

   \item $R(B_{n-1},B_n) \le 4n-1$, \\
   \item  $R(B_{n-2},B_n) \le 4n-3 \text{ if } n \equiv 2 \pmod 3$. 
\end{enumerate}
\end{theorem}

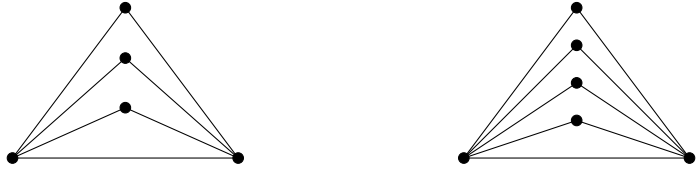
\begin{figure}
    \centering
    \begin{tikzpicture}
  \filldraw (0,0) circle (2pt);
  \filldraw (3,0) circle (2pt);

    \filldraw (1.5,2) circle (2pt);
    \filldraw (1.5,1.33) circle (2pt);
    \filldraw (1.5,.67) circle (2pt);
    
    \draw (0,0) -- (3,0);
    \draw (0,0) -- (1.5,2);
    \draw (0,0) -- (1.5,1.33);
    \draw (0,0) -- (1.5,.67);
    \draw (3,0) -- (1.5,2);
    \draw (3,0) -- (1.5,1.33);
    \draw (3,0) -- (1.5,.67);

\filldraw (6,0) circle (2pt);
  \filldraw (9,0) circle (2pt);

    \filldraw (7.5,2) circle (2pt);
    \filldraw (7.5,1.5) circle (2pt);
    \filldraw (7.5,1) circle (2pt);
    \filldraw (7.5,.5) circle (2pt);
    \draw (6,0) -- (9,0);
    \draw (6,0) -- (7.5,2);
    \draw (6,0) -- (7.5,1.5);
    \draw (6,0) -- (7.5,1);
    \draw (6,0) -- (7.5,0.5);
    \draw (9,0) -- (7.5,2);
    \draw (9,0) -- (7.5,1.5);
    \draw (9,0) -- (7.5,1);
    \draw (9,0) -- (7.5,0.5);
\end{tikzpicture}
    \caption{Book graphs $B_3$ and $B_4$.}
    \label{Figure_books}
\end{figure}

The upper bounds in Theorem \ref{TheoremAlmostDiagonalBookUpperBound} come from careful examination of Goodman's bound for the number of monochromatic triangles in a 2-edge-coloring of a complete graph \cite{GoodmanRamsey}. The lower bounds for $R(B_n,B_n)$ come from Paley graphs. Further bounds on small Ramsey numbers for books and other graphs were found in \cite{BlackLevenRadz,ShaoXuBoPan,FaudreeRousseauSheehanStronglyRegular,LidickyPfenderSDPRamsey} (see \cite{RamseySurvey} for a detailed accounting). 

In recent years, SAT solvers have been instrumental in solving difficult combinatorial problems. In particular, they have been used effectively to compute exact bounds in arithmetic Ramsey theory, namely in Schur, Rado, and van der Waerden numbers \cite{SchurFive, VDW26, VDW34, PythagoreanTriplesSAT, BMRS_3ColorSchur, WJW_Rado_ISSAC}. In 2016, Codish, Frank, Itzhakov, and Miller computed the second known three-color Ramsey number, $R(3,3,4) = 30$, with SAT solvers \cite{R334Equals30}. More recent works have used SAT solvers to certify the bound $R(3,8) \le 28$ \cite{R38_Verify} and compute the bound $R(5,5) \le 46$ \cite{R55Le46}.

The main contributions of our paper are improvements on book Ramsey number lower bounds using \emph{block-circulant} graphs, which have yielded recent improvements for Ramsey numbers involving the graphs $K_{m,n}, K_n - e,$ and the wheel graphs $W_n$  \cite{BlockCircRamseyGoedVanOver}. In particular, we show that a certain generalization of the Paley graph gives the tight bound $R(B_{n-1},B_n) = 4n-1$ for infinitely many $n$. We also produce block-circulant graphs that give new lower bounds for $R(B_r,B_s)$ for several small values of $r$ and $s$, some of which are tight as well. Moreover, we obtain additional lower and upper bounds for $R(B_r,B_s)$ using SAT and integer programming (IP) solvers. Using the related SAT modulo symmetries (SMS) framework \cite{SMS}, we are able to enumerate the graphs avoiding copies of $B_r$ and $\overline B_s$ for some small values of $r$ and $s$. 

 We state our main results precisely in Section \ref{SectionResults}. In Section \ref{SectionBlockCirculant} we give more details on block-circulant graphs and our Paley-type constructions. Section \ref{SectionSAT} outlines the computational methods to achieve the improved bounds. Data for graphs that achieve the lower bounds is given in the Appendix. 

\section{Results}\label{SectionResults}
Our main results are the following. First, we show that the upper bound $R(B_{n-1},B_n) \le 4n-1$ given in Theorem \ref{TheoremAlmostDiagonalBookUpperBound} is tight for an infinite family of $n$. 

\begin{theorem}\label{TheoremAlmostDiagPrimePowers}
    The book Ramsey number $R(B_{n-1},B_n)$ equals $4n-1$ for $n \le 20$. Moreover, if $2n-1$ is a prime power congruent to $1$ modulo $4$, then $R(B_{n-1},B_n) = 4n-1$. 
\end{theorem}

We also computed the following new bounds for other values $R(B_r,B_s)$. We note that some of these bounds, namely $R(B_2,B_8) \ge 21, R(B_2,B_9) \ge 22,$ and $R(B_5,B_7) \ge 25$, have been computed independently using tabu search methods in the recent paper by Lidick\'y, McKinley, Pfender, and Van Overberghe \cite{LidickyMcKinleyPfenderSmallBooksWheels}. The bound $R(B_3,B_6) \ge 19$ also appeared in the author's Ph.D. thesis \cite{WJWThesis}. 

\begin{theorem}\label{ThemoremMiscBounds}
The values for $R(B_r,B_s)$ in the following tables hold. 


\begin{table}[h]
\caption{$R(B_r,B_s)$ for $r = 2,3$} \label{Table23}
\begin{center}
    \begin{tabular}{|c|c|c|c|c}
\hline
     $r$ & $s$ & $R(B_r,B_s)$ \\
    \hline 
   $2 $&$ 8 $&$ \bf{21} $\\
    $2 $&$ 9 $&$ \bf{22} $\\
    $2 $&$ 10 $&$ \bf{25} $\\
    $2 $&$ 12 $&$ \bf{28} $\\
    $2$ & $13$ & $\bf{29}$ \\
    $3 $&$ 6 $&$ \bf{19} $\\
    $3 $&$ 7 $&$ \bf{20} $\\
    \hline
\end{tabular}
\end{center}
\end{table}
\begin{table}[h] 
\caption{$R(B_r,B_s)$ for $r = s-2$}\label{TableAlmostDiag}
\begin{center}
\begin{tabular}{|c|c|c|c|c}
\hline
     $r$ & $s$ & $R(B_r,B_s)$ \\
    \hline 
    $5 $&$ 7 $& $\ge \bf{25}$ \\ 
    $6 $&$ 8 $&$ \bf{29}  $\\
    $7 $&$ 9 $&$ \ge \bf{33}$ \\
    $8 $&$ 10 $&$ \ge \bf{37} $\\
    $9 $&$ 11 $&$ \bf{41}  $\\
    $10 $&$ 12 $&$ \ge\bf{45}  $\\
    $11 $&$ 13 $&$ \ge \bf{49}  $\\
    $12 $&$ 14 $&$ \bf{53}  $\\
    $13 $&$ 15 $&$ \ge \bf{57}  $\\ 
    $14 $&$ 16 $&$ \ge \bf{61}   $\\
    $15 $&$ 17 $&$ \bf{65}  $\\ 
    \hline
\end{tabular}
\end{center}
\end{table}
\begin{table}[h]
\caption{$R(B_r,B_s)$ for $r = s$}\label{TableDiag}
\begin{center}
\begin{tabular}{|c|c|c|c|c}
\hline
     $r$ & $s$ & $R(B_r,B_s)$ \\
    \hline 
    $8 $&$ 8 $&$ \bf{33} $\\
    $11$ & $11$ & $\ge \bf{45}$ \\ 
    $14$ & $14$ & $\bf{57}$ \\
    $16$ & $16$ & $\ge \bf{65}$\\
    \hline 
\end{tabular}
\end{center}
\end{table}

\end{theorem}

A \emph{Ramsey $(G,H,n)$ graph} is a graph on $n$ vertices that does not contain a copy of $G$ or $\overline{H}$. We say that such a graph is a \emph{critical graph} if $n = R(G,H)-1$. It is of general interest to enumerate the critical graphs (up to isomorphism) for given $G$ and $H$, and this was done for $(G,H) \in \{(B_2,B_3),(B_3,B_4),(B_2,B_6),(B_2,B_7)\}$ in \cite{ShaoXuBoPan,BlackLevenRadz}. The following theorem gives the numbers of critical graphs for $G = B_r$, $H = B_s$ and small values of $r$ and $s$.  

\begin{theorem} \label{TheoremEnumeration}
The following hold. 
\begin{center}
        \begin{tabular}{|c|c|c|c|}
            \hline
            $r$ & $s$ & $R(B_r,B_s)$ &  \#critical graphs\\
            \hline
            $1$ & $1$ & $6$ & $1$\\  
            $1$ & $2$ & $7$ & $4$ \\ 
            $1$ & $3$ & $9$ & $8$ \\
            $1$ & $4$ & $11$ & $7$ \\
            
            $1$ & $5$ & $13$ & $8$ \\ 
            $1$ & $6$ & $15$ & $8$ \\
$1$ & $7$ & $17$ & $10$ \\ 
$1$ & $8$ & $19$ & $10$  \\
$2$ & $2$ & $10$ & $1$\\
$2$ & $3$ & $11$ & $4$ \cite{ShaoXuBoPan}\\
 $2$ & $4$ & $13$ & $6$  \\ 
  $2$ & $5$ & $16$ & $1$ \\
  $2$ & $6$ & $17$ & $3$ \cite{BlackLevenRadz}  \\
  $2$ & $7$ & $18$ & $65$ \cite{BlackLevenRadz}  \\
    $        2 $&$ 8 $&$ 21 $&$ 1 $ \\
     $       2 $&$ 9 $&$ 22 $&$ 72 $ \\
      $      2 $&$ 10 $&$ 25 $&$ 5 $ \\
      $2$ & $11$ & $28$ & $1$  \\
      $2$ & $12$ & $28$ & $10$\\
\hline 
\end{tabular}
\begin{tabular}{|c|
c|c|c|} 
     \hline 
         $r$ & $s$ & $R(B_r,B_s)$ &  \#critical graphs\\
            \hline
      $3$ & $3$ & $14$ & $1$  \\ 
       $3 $&$ 4 $&$ 15 $&$ 1 $ \cite{ShaoXuBoPan} \\
      $3$ & $5$ & $17$ & $10$ \\
       $     3 $&$ 6 $&$ 19 $&$ 4 $\\
       $4$ & $4$ & $17$ & $1$  \\
        $          4 $&$ 5 $&$ 19 $&$ 27 $ \\
       $5$ & $5$ & $21$ & $247$  \\
        $         5 $&$ 6 $&$ 23 $&$ 23 $ \\
       $6$ & $6$ & $26$ & $15$ \\
 \hline 
 \multicolumn{1}{c}{}\\
 \multicolumn{1}{c}{}\\
 \multicolumn{1}{c}{}\\
 \multicolumn{1}{c}{}\\
 \multicolumn{1}{c}{}\\
 \multicolumn{1}{c}{}\\
 \multicolumn{1}{c}{}\\
 \multicolumn{1}{c}{}\\
\multicolumn{1}{c}{}\\
\multicolumn{1}{c}{}\\
        \end{tabular}
        \end{center}
\end{theorem}

\section{Block-circulant Graphs}\label{SectionBlockCirculant}

Recall that a \emph{circulant matrix} is a matrix such that each row is a one-element right cyclic shift of the previous row. A \emph{circulant graph} is a graph whose adjacency matrix is circulant. More generally, a graph is \emph{block-circulant} on $k$ blocks (or simply \emph{$k$-block-circulant}) if its adjacency matrix can be written in the form $$A = \begin{pmatrix}C_{11} & C_{12} & \dots & C_{1k} \\ C_{21}  & C_{22} & \dots & C_{2k} \\ \vdots & \vdots & \ddots & \vdots 
& \\ C_{k1} & C_{k2} & \dots & C_{kk}\end{pmatrix}, $$ where $C_{ij}$ is a circulant matrix for all $i,j$. Observe that since $A$ is symmetric, we have $C_{ij} = C_{ji}^T$ for all $i$ and $j$. Note also that $C_{ij}$ must be symmetric for $i = j$, but $C_{ij}$ need not be symmetric for $i \neq j$. A block-circulant graph, therefore, is determined completely by the first rows of $C_{ij}$ for $i\le j$. For all $i,j$, let $D_{ij}$ denote the indices (using zero-indexing) of the columns of $C_{ij}$ whose first entry is equal to 1. Since $A$ is an adjacency matrix, we must have $0 \not \in D_{ii}$ for all $i$, but it is possible that $0 \in D_{ij}$ for $i \neq j$. We identify the tuple $(D_{ij})_{i\le j}$ with the corresponding block-circulant graph. 

Circulant graphs often provide good (and in several cases, tight) lower bounds for the classical Ramsey numbers $R(s,t)$. The more general block-circulant graphs have also yielded improved lower bounds for other Ramsey numbers, in particular for $R(K_{s} -e, K_{t}-e)$ \cite{BlockCircRamseyGoedVanOver}. Here we show that block-circulant graphs give tight bounds for Ramsey numbers for book graphs. 

We will use the following notation throughout this paper. 

\begin{definition}
    For subsets $X,Y$ of an additive group $G$ and $d \in G$, let
    \begin{align*}
    \Delta(X,Y,d) &:= |\{(x,y) \in X \times Y : x-y = d\} |,\\ \Sigma(X,Y,d) &:= |\{(x,y) \in X \times Y : x+y = d\} |.
\end{align*}
\end{definition}

\begin{definition}
    Given a graph $G = (V,E)$ and two vertices $u,v \in V$, let $$\Gamma(u,v) := \{ w \in V : \{u,w\} \in E \text{ and } \{v,w\} \in E \}.$$
\end{definition}

A useful property of 2-block-circulant graphs is that $|\Gamma(u,v)|$ can be computed easily via the following lemma. 

\begin{lemma}\label{LemmaNumberCommonNeighbors2Block}
    Let $G = (V,E)$ be a $2$-block-circulant graph $(D_{11},D_{12},D_{22})$ with $V = V_1 \sqcup V_2$ and $V_1 = V_2 = \Z_m$. That is, the vertex set is the disjoint union of two copies of $\Z_m$ and $|V| = 2m$. Then for all $u,v \in V$, we have 
    \begin{align*}|\Gamma(u,v)| =  \begin{cases} 
      \Delta(D_{11},D_{11},v-u) +  \Delta(D_{12},D_{12},v-u)  \text{ if }  {u,v \in V_1} \\ 
    \Delta(D_{22},D_{22},v-u) +  \Delta(D_{12},D_{12},v-u) \text{ if } {u,v \in V_2} \\
    \Sigma(D_{11},D_{12},v-u) +  \Delta(D_{12},D_{22},v-u) \text{ if } {u \in V_1, v \in V_2}
    \end{cases} 
    \end{align*}
\end{lemma}
\begin{proof} 
If $u \in V_i$ and $v \in V_j$ with $i\le j$, then $u$ and $v$ are adjacent if and only if $v-u \in D_{ij}$. 
Suppose $u,v$ belong to $V_1$. Then they have a common neighbor $w$ if and only if $w-u, w-v \in D_{11}$ or $w-u, w-v \in D_{12}$. Then, letting $\chi_S$ denote the indicator function of a set $S$, we have  
\begin{align*}
    |\Gamma(u,v)| &= \sum_{w \in V_1} \chi_{D_{11}}(w-u)\chi_{D_{11}}(w-v) + \sum_{w \in V_2} \chi_{D_{12}}(w-u)\chi_{D_{12}}(w-v) \\
    &=\sum_{w+v \in V_1} \chi_{D_{11}}(w+v-u)\chi_{D_{11}}(w) + \sum_{w+v \in V_2} \chi_{D_{12}}(w+v-u)\chi_{D_{12}}(w) \\
    &= \sum_{w \in V_1} \chi_{D_{11}}(w+v-u)\chi_{D_{11}}(w) + \sum_{w \in V_2} \chi_{D_{12}}(w+v-u)\chi_{D_{12}}(w) \\
    &= \Delta(D_{11},D_{11},v-u) +\Delta(D_{12},D_{12},v-u).
\end{align*}

The case where $u,v \in D_{22}$ is similar. If $u \in V_1, v \in V_2$, then $u$ and $v$ have a common neighbor $w$ in $V_1$ if and only if $w - u \in D_{11}$ and $v-w \in D_{12}$, and a common neighbor $w \in V_2$ if and only if $w-u \in D_{12}$ and $w-v \in D_{22}$. The total number of neighbors of $u$ and $v$ is then $\Sigma(D_{11},D_{12},v-u) +  \Delta(D_{12},D_{22},v-u)$. 

Then we have 

\begin{align*}
    |\Gamma(u,v)| &= \sum_{w \in V_1} \chi_{D_{11}}(w-u)\chi_{D_{12}}(v-w) + \sum_{w \in V_2} \chi_{D_{12}}(w-u)\chi_{D_{22}}(w-v) \\
    &=\sum_{w+v \in V_1} \chi_{D_{11}}(w+v-u)\chi_{D_{12}}(-w) + \sum_{w+v \in V_2} \chi_{D_{12}}(w+v-u)\chi_{D_{22}}(w) \\
    &= \sum_{w \in V_1} \chi_{D_{11}}(w+v-u)\chi_{D_{12}}(-w) + \sum_{w \in V_2} \chi_{D_{12}}(w+v-u)\chi_{D_{22}}(w) \\
    &= \Sigma(D_{11},D_{12},v-u) +\Delta(D_{12},D_{22},v-u).
\end{align*}

\end{proof} 
The following corollary gives a lower bound for $R(B_r,B_s)$. It follows quickly from the preceding lemma and the definition of book graphs. 
\begin{corollary}\label{Corollary2BlockLB}
    Let $G = (D_{11},D_{12},D_{22})$ be a $2$-block-circulant graph on $2m$ vertices. Let $\overline{D}_{ii}, \overline{D}_{12}$ denote, respectively, the complement of $D_{ii}$ in $\Z_m \setminus \{0\}$ and the complement of $D_{12}$ in $\Z_m$. If 
    \begin{align*}
         \Delta(D_{11},D_{11},d) +  \Delta(D_{12},D_{12},d)  < r, \text{ for all $d \in D_{11}$}  \\ 
    \Delta(D_{22},D_{22},d) +  \Delta(D_{12},D_{12},d) < r, \text{ for all $d \in D_{22}$}  \\
    \Sigma(D_{11},D_{12},d) +  \Delta(D_{12},D_{22},d) < r, \text{ for all $d \in D_{12}$} \\
      \Delta(\overline D_{11},\overline D_{11},d) +  \Delta(\overline D_{12},\overline D_{12},d)  < s, \text{ for all $d \in \overline D_{11}$}  \\ 
    \Delta(\overline D_{22},\overline D_{22},d) +  \Delta(\overline D_{12},\overline D_{12},d) < s, \text{ for all $d \in \overline D_{22}$}  \\
    \Sigma(\overline D_{11},\overline D_{12},d) +  \Delta(\overline D_{12},\overline D_{22},d) < s, \text{ for all $d \in \overline D_{12}$}
    \end{align*}

    then $R(B_r, B_s) > 2m$.
\end{corollary}
\begin{proof}
    A graph $G$ does not contain a copy of $B_r$ if and only if $|\Gamma(u,v)| < r$ for adjacent $u,v$. The first three hypotheses correspond to the cases $u,v \in V_1$, $u,v \in V_2$, and $u \in V_1, v \in V_2$, respectively. By Lemma \ref{LemmaNumberCommonNeighbors2Block}, this implies that $|\Gamma(u,v)| < r$ for adjacent $u,v$. Similarly, the last three hypotheses and Lemma \ref{LemmaNumberCommonNeighbors2Block} imply that $G$ contains no copy of $\overline{B_s}$, and the result follows. 
\end{proof}

More generally, we can make the following construction, which is a ``2-block" Cayley graph. Let $G$ be an abelian group, and let $D_{11},D_{12},D_{22} \subseteq G$. Define a graph $\Gamma_G(D_{11},D_{12},D_{22})$ with vertex set $V_1 \sqcup V_2$ with $V_1 = V_2 = G$ such that $x,y$ are adjacent if and only if $x,y \in V_1$ and $y-x \in D_{11}$, $x,y \in V_2$ and $y-x \in D_{22}$, or $x \in V_1, y \in V_2$ and $y-x \in D_{12}$. A 2-block-circulant graph on $2m$ vertices is the special case $G = \Z_m$.

\begin{lemma}\label{LemmaBookDifferenceSetConditions}
    Let $G$ be an abelian group. Let $\overline{D}_{ii}, \overline{D}_{12}$ denote, respectively, the complement of $D_{ii}$ in $G \setminus \{0\}$ for $i = 1,2$ and the complement of $D_{12}$ in $G$. If 
    \begin{align*}
         \Delta(D_{11},D_{11},d) +  \Delta(D_{12},D_{12},d)  < r, \text{ for all $d \in D_{11}$}  \\ 
    \Delta(D_{22},D_{22},d) +  \Delta(D_{12},D_{12},d) < r, \text{ for all $d \in D_{22}$}  \\
    \Sigma(D_{11},D_{12},d) +  \Delta(D_{12},D_{22},d) < r, \text{ for all $d \in D_{12}$} \\
      \Delta(\overline D_{11},\overline D_{11},d) +  \Delta(\overline D_{12},\overline D_{12},d)  < s, \text{ for all $d \in \overline D_{11}$}  \\ 
    \Delta(\overline D_{22},\overline D_{22},d) +  \Delta(\overline D_{12},\overline D_{12},d) < s, \text{ for all $d \in \overline D_{22}$}  \\
    \Sigma(\overline D_{11},\overline D_{12},d) +  \Delta(\overline D_{12},\overline D_{22},d) < s, \text{ for all $d \in \overline D_{12}$}
    \end{align*}

    then $\Gamma_G(D_{11},D_{12},D_{22})$ does not contain a copy of $B_r$ or $\overline{B}_s$, hence $R(B_r, B_s) > 2|G|$.
\end{lemma}
The proof of Lemma \ref{LemmaBookDifferenceSetConditions} is essentially the same as Lemma \ref{LemmaNumberCommonNeighbors2Block} and Corollary \ref{Corollary2BlockLB}; we omit the details here.

Recall that in the Paley graph of order $q$, vertices $x$ and $y$ are adjacent if and only if $x-y$ is a nonzero quadratic residue in $\F_q$. In particular, for prime $q$, Paley graphs are circulant. In Theorem \ref{TheoremAlmostDiagonalBookUpperBound}, the lower bound $R(B_n,B_n) \ge 4n+2$ when $q = 4n+1$ is a prime power comes from the fact that the Paley graph of order $q$ does not contain copies of $B_n$ and is self-complementary. More precisely, two adjacent vertices share exactly $n-1$ common neighbors. Our construction in the proof of Theorem \ref{TheoremAlmostDiagPrimePowers} is a generalization of the Paley graphs using two blocks instead of one.  The following lemma gives convenient properties of quadratic residues and nonresidues in $\F_q$. 
\begin{lemma} \label{LemmaPrimePowerResidueDifferences}
    Let $q \equiv 1 \pmod 4$ be a prime power. Let $Q$ and $N$ denote the set of nonzero quadratic residues and nonresidues, respectively, in $\F_q$. Then 

    \begin{align*}
        \Delta(Q,Q,d) = \begin{cases}
        \frac{q-1}4 -1 & d \in Q \\
        \frac{q-1}4 & d \in N
        \end{cases}\\
        \Delta(N,N,d) = \begin{cases}
        \frac{q-1}4  & d \in Q \\
        \frac{q-1}4 -1 & d \in N
        \end{cases}\\
        \Delta(Q,N,d) = \Delta(N,Q,d) =  \begin{cases}
        \frac{q-1}4  & d \neq 0\\
        0 & d = 0
        \end{cases}
    \end{align*}
\end{lemma}

\begin{proof}
    Let $S,S' \in \{Q,N\}$, and let $\chi_{A}$ denote the indicator function of a set $A$. Observe that the analogue of Euler's criterion holds in $\F_q$, that is, for all $x \in \F_q$, we have $$x^{(q-1)/2} = \begin{cases}
        1 & \text{if } x \in Q \\
        -1 & \text{if } x \in N \\
        0 & \text{if } x = 0.
    \end{cases} $$
    
    Recall also that exactly half of the elements of $\F_q \setminus \{0\}$ are squares, so $\sum_{x \in \F_q} x^{(q-1)/2} = 0$.
    
   Let $\delta_0(x) = \chi_{\{0\}}(x)$. Let $\epsilon = \begin{cases} 1 & \text{ if } S = Q \\ -1 & \text{ if } S = N \end{cases}$, and similarly for $\epsilon'$ and $S'$.
   
   Then 

 \begin{align*}
    \Delta(S,S',d) 
    &= \sum_{x \in \F_q}\chi_{S}(x+d) \chi_{S'}(x)  \\
    &= \sum_{x \in \F_q} \left( \frac{1+\epsilon(x+d)^{(q-1)/2} -\delta_0(x+d)}{2}\right)\left( \frac{1+\epsilon'x^{(q-1)/2} -\delta_0(x)}{2}\right)\\
    &= \frac 14\sum_{x \in \F_q} (1 -\delta_0(x) - \delta_0(x+d)+\epsilon'x^{(q-1)/2}  +\epsilon(x+d)^{(q-1)/2} -\epsilon\delta_0(x)(x+d)^{(q-1)/2} \\&\quad -\epsilon'\delta_0(x+d) x^{(q-1)/2} + \epsilon \epsilon'(x+d)^{(q-1)/2}x^{(q-1)/2} + \delta_0(x)\delta_0(x+d) ) \\
    &= \frac 14\Bigl(q-2 -\epsilon d^{(q-1)/2} -\epsilon' (-d)^{(q-1)/2} + \epsilon \epsilon'\sum_{x \in \F_q}(x^2+xd)^{(q-1)/2} \Bigr) \\
    &= \frac 14\Bigl(q-2 -(\epsilon + \epsilon') d^{(q-1)/2} + \epsilon \epsilon'\sum_{x \in \F_q \setminus \{0\} } (x^2+xd)^{(q-1)/2} \Bigr)\\
    &= \frac 14\Bigl(q-2 -(\epsilon + \epsilon') d^{(q-1)/2} + \epsilon \epsilon'\sum_{x \in \F_q \setminus \{0\} } (1+x\inv d)^{(q-1)/2}\Bigr) \\
    &= \frac 14\Bigl(q-2 -(\epsilon + \epsilon') d^{(q-1)/2} + \epsilon \epsilon'\sum_{y \in \F_q \setminus \{1\} } y^{(q-1)/2}\Bigr) \\
    &= \frac 14\Bigl(q-2 -(\epsilon + \epsilon') d^{(q-1)/2} - \epsilon \epsilon'\Bigr).
    \end{align*}
    
\end{proof}

We can now prove Theorem \ref{TheoremAlmostDiagPrimePowers}. 
\begin{proof}
    Let $q = 2n-1$ be a prime power equivalent to 1 modulo 4. Let $Q$ and $N$ denote respectively the set of quadratic residues and nonresidues in $\F_q$. Set $D_{11} = D_{12} = Q$, and $D_{22}  = N$. We will show that the conditions in Lemma \ref{LemmaBookDifferenceSetConditions} hold for the graph $\Gamma_{\F_q}(Q,Q,N)$.

 Since $q \equiv 1 \pmod 4$, we have $-1 \in Q$. It follows that $d \in Q$ if and only if $-d \in Q$, and also $d \in N$ if and only if $-d \in N$. Therefore $\overline{D}_{22} = D_{11} = Q$, $\overline{D}_{11} = D_{22} = N$, and $\overline{D}_{12} = N \cup \{0\}$. 

By Lemma \ref{LemmaPrimePowerResidueDifferences}, we have the following. 

For all $d \in D_{11} = Q$, $$\Delta(D_{11},D_{11},d) + \Delta(D_{12},D_{12},d) = 2 \Delta(Q,Q,d) = \frac{q-1}2 - 2 = n-3.$$
 
For all $d \in D_{22} = N$,
$$\Delta(D_{22},D_{22},d) + \Delta(D_{12},D_{12},d) = \frac{q-1}4 - 1+ \frac{q-1}4 =  n-2.$$
 
For all $d \in D_{12} = Q$, 
\begin{align*}
    \Sigma(D_{11},D_{12},d) +  \Delta(D_{12},D_{22},d) &= \Delta(D_{11},-D_{12},d) + \Delta(D_{12},D_{22},d) \\&= \Delta(Q,Q,d) + \Delta(Q,N,d) \\&= \frac{q-1}4 - 1 + \frac{q-1}4 \\&= n-2. 
\end{align*}

For all $d \in \overline D_{11} = N$,   \begin{align*}\Delta(\overline D_{11},\overline D_{11},d) +  \Delta(\overline D_{12},\overline D_{12},d) &= 
  \Delta(N,N,d) +  \Delta(N \cup \{0\},N \cup \{0\},d) \\& = \Bigl (\frac{q-1}4 -1 \Bigr) + \Bigl( \frac{q-1}4 +1 \Bigr) \\& = n-1,
\end{align*}

For all $d \in \overline{D}_{22} = Q$,

\begin{align*}
    \Delta(\overline D_{22},\overline D_{22},d) +  \Delta(\overline D_{12},\overline D_{12},d) &= \Delta(Q,Q,d) +  \Delta(N \cup \{0\} ,N \cup \{0\} ,d) \\& = \frac{q-1}4 -1 + \frac{q-1}4 = n-2
\end{align*}

For all $d \in \overline{D}_{12} = N \cup \{0\} $,

\begin{align*}
    \Sigma(\overline D_{11},\overline D_{12},d) +  \Delta(\overline D_{12},\overline D_{22},d)   &=  \Delta(\overline D_{11},-\overline D_{12},d) +  \Delta(\overline D_{12},\overline D_{22},d) \\
    &= \Delta(N,N \cup \{0\},d) +  \Delta(N\cup \{0\} ,Q    ,d) 
    \\
    &= \begin{cases}
        \frac{q-1}2 + 0 \text{ if } d = 0 \\
        \frac{q-1}{4} + \frac{q-1}{4} \text { if } d \neq 0 \\
    \end{cases}\\
     & \le n-1.
\end{align*}

    By Lemma \ref{LemmaBookDifferenceSetConditions}, $\Gamma_{\F_q}(Q,Q,N)$ does not contain a copy of $B_{n-1}$ or $\overline {B_n}$. Therefore $R(B_{n-1},B_n) \ge 2q+1 = 4n-1$, and by Theorem \ref{TheoremAlmostDiagonalBookUpperBound}, this is tight.  
\end{proof}
The above proof gives the bound $R(B_{n-1},B_n) \ge 4n-1$ for an infinite family of $n$. For the small values $n \ge 20$ and the bounds for other Ramsey numbers in Theorem \ref{ThemoremMiscBounds}, we outline the computational methods used to obtain them in the next section.

\section{SAT and IP methods}\label{SectionSAT}
There are natural encodings for computing Ramsey numbers using SAT solvers and IP solvers. For arbitrary graphs $G_1$ and $G_2$, the following encoding gives a way to compute $R(G_1,G_2)$.

\begin{encoding}\label{EncodingRamseySAT}
	The Ramsey number $R(G_1,G_2)$ is at most $n$ if the formula $F_n(G_1,G_2)$ is unsatisfiable, where and 
	$$F_n(G_1,G_2) :=\Bigl(\bigwedge_{H\subset K_n, H\cong G_1} \Bigl(\bigvee_{e\in E(H)} \bar x_e\Bigr) \Bigr)\wedge \Bigl(\bigwedge_{H\subset K_n, H\cong G_2} \Bigl(\bigvee_{e\in E(H)} x_e\Bigr)\Bigr).$$
	Moreover, if $F_n(G_1,G_2)$ is satisfiable, then $R(G_1,G_2) \ge n+1$.
\end{encoding}

However, for $(G_1,G_2) = (B_r,B_s)$, this formula becomes too large even for modest values of $r$ and $s$ because $K_n$ contains $\binom n 2 (\binom {n-2}{r}+\binom{n-2}{s})$ copies of $B_r$ and $B_s$. To obtain a smaller encoding, we introduce new variables $y_{ijk}$ and $y'_{ijk}$ for each triangle, letting
\begin{equation}\label{Eq_y}
    y_{ijk} := x_{ij} \wedge x_{ik} \wedge x_{jk},\quad  y'_{ijk} := \bar{x}_{ij} \wedge \bar{x}_{ik} \wedge \bar{x}_{jk}. 
\end{equation}

(Note that here we ignore the order of the subscript indices, so that $y_{123}$ and $y_{321}$ are the same.) 

Now let $L_r(i,j)$ (respectively $L'_r(i,j)$) denote the constraint that at most $r$ of the variables $y_{ijk}$ (respectively $y'_{ijk}$) are set to true. Then we have the following encoding to bound $R(B_r,B_s)$.

\begin{encoding} \label{EncodingSATBooks}
    With notation as above, the Ramsey number $R(B_r,B_s)$ is at most $n$ if and only if the formula 

    $$F_n(r,s) := \bigwedge_{1 \le j \le n } L_r(i,j) \cap \bigwedge_{1 \le j \le n } L_s'(i,j)$$ is unsatisfiable.
\end{encoding}
This encoding can be converted to conjunctive normal form (CNF) by using a Tseitin transformation for the formulas in \eqref{Eq_y} and using a cardinality constraint encoding for the $L_r(i,j)$ and $L'_s(i,j)$. For this latter encoding we used the tree encoding from \cite{WynnSATcardinality}, Section 2.4 (see also \cite{BailleuxBoufkhadTreeEncoding}). 

Finally, we added symmetry breaking clauses to enforce that no simple vertex transposition (that is, swapping vertices $i$ and $i+1$ for some $i$) results in a lexicographically smaller adjacency matrix. This is the same set of clauses produced by the SAT symmetry breaking software {\scshape Shatter} \cite{Shatter}. We could have introduced more symmetry breaking clauses, for instance other vertex transpositions, but in practice this usually does not provide any useful speedup. 

Another popular method for computing Ramsey number bounds is integer programming (IP). For some of our lower bounds, this method was more efficient than SAT solving. In particular, IP solving worked well for finding 2-block-circulant graphs (though it is not difficult to enforce block-circulant constraints in the SAT encoding). For our book graphs, we can encode the problem as an integer program by converting the constraints in Corollary \ref{Corollary2BlockLB} into linear inequalities. For brevity, we will not give the encoding in full detail here, but as an example, the constraint $$\Delta(D_{11},D_{11},d) +  \Delta(D_{12},D_{12},d)  < r, \text{ for all } d \in D_{11}$$ in Corollary \ref{Corollary2BlockLB} can be encoded as follows: 

\begin{align*}
    s_{i,d} &\le x_i,  && u_{i,d} \le z_i,\\
    s_{i,d} &\le x_{i+d}, &&  u_{i,d} \le z_{i+d}, \\
    s_{i,d} &\le x_d, &&  u_{i,d} \le z_d,\\
    s_{i,d} &\ge x_i + x_{i+d} + x_d -2, &&u_{i,d} \ge z_i + z_{i+d} + z_d -2, \quad \quad \forall i,d \in \Z_m \\
    \sum_{i \in \Z_m } & s_{i,d} + u_{i,d} \le r-1  && \forall d \in \Z_m, \\
    x_0 &= 0, \\ 
    x_i, &z_i, s_{i,d}, u_{i,d} \in \{0,1\} && \forall i,d \in \Z_m. 
\end{align*}

Here the variables $x_i$ are set to 1 if $i \in D_{11}$, and similarly $z_i$ is set to 1 if $i \in D_{12}$. Then $s_{i,d} = 1$ if and only if $d \in D_{11}$ and $x_i, x_{i+d} \in D_{11}$, and similarly for $z_{i,d}$. 



\subsection{Proofs of Theorem \ref{ThemoremMiscBounds} and Theorem \ref{TheoremEnumeration}}

The bounds in Theorem \ref{ThemoremMiscBounds} and the bounds for $n \le 20$ in Theorem \ref{TheoremAlmostDiagPrimePowers} were computed using the SAT solver {\scshape{Kissat}} \cite{Kissat} and the IP solver {\scshape{SCIP}} \cite{SCIP6}. Generally speaking, the lower bounds in Table \ref{Table23} were produced using SAT solving and those in Tables \ref{TableAlmostDiag} and \ref{TableDiag} were produced by integer programming. However, there are two exceptions: the lower bounds $R(B_2,B_{12}) \ge 28$ and $R(B_2, B_{13})$ follow from results in \cite{FaudreeRousseauSheehanStronglyRegular}. Namely, the Schl\"afli graph is a 27-vertex graph that contains no copy of $B_{11}$ and its complement contains no $B_2$, so $28 \le R(B_2, B_{11}) \le R(B_2,B_{12})$. Moreover, the complete bipartite graph $K_{14,14}$ contains no $B_2$ and its complement contains no $B_{13}$, so $R(B_2,B_{13}) \ge 29$. For each lower bound, we produced a graph avoiding copies of $B_r$ and $\overline{B}_s$. These graphs can be found in the Appendix, and we display two such graphs in Figure \ref{Figure_2blockCircAlmostDiag}. 
\begin{figure} 
\begin{center}
     \includegraphics[scale=0.3]{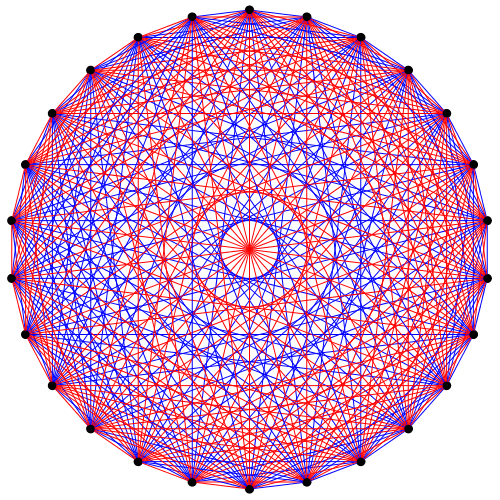}
    \includegraphics[scale=0.3]{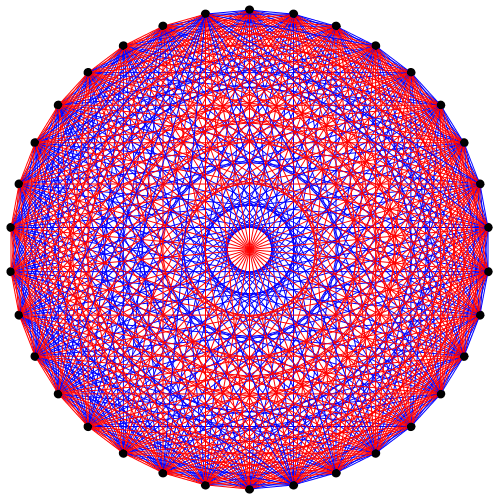}
\end{center}
\caption{2-block-circulant critical graphs for $R(B_6,B_7)$ (left) and $R(B_8,B_9)$ (right). Present edges are colored red, missing edges are colored blue.}\label{Figure_2blockCircAlmostDiag}
\end{figure}
In each of the 2-block-circulant graphs we initially found, we observed that $D_{11}$ and $D_{22}$ were complements (with the exception of 0). We added this constraint as an ansatz for higher values of $r$ and $s$, and this alone never resulted in an infeasible program. In our Paley-type constructions we also had $D_{11} = D_{12}$, but enforcing this constraint did give infeasible programs for some values of $r$ and $s$. For higher values of $r$ and $s$ we did add constraints of the form $S \subset D_{11}$ for certain sets $S$, for example $S = \{1,2,3,4\}$ and $S = \{1,4,9,16\}$, and this sped up computation in some cases. Our computational evidence suggests that $R(B_{n-1},B_n) = 4n-1$ for all $n$, and we conjecture this is true, though unfortunately, we were unable to find a general construction.

 For the new upper bounds $R(B_2,B_8) \le 21$, $R(B_2,B_9) \le 22$, and $R(B_3,B_7) \le 20$, we used Encoding \ref{EncodingSATBooks} and showed the corresponding formula was unsatisfiable. For each bound we produced DRAT certificates, sequences of clauses (and deletions) that are redundant with respect to the formula. These certificates are then used to verify the computation by the SAT solver is correct. The most difficult were by far $R(B_3,B_7) \le 20$ and $R(B_2,B_{13}) \le 29$, which yielded DRAT certificates of sizes approximately 19GB and 186GB, respectively. 
 

\begin{center}
\begin{tabular}{|c|c|c|c|}
\hline 
    $r$ & $s$ & $R(B_r,B_s)$ & Upper bound solve time (seconds) \\
    \hline
     2& 8 & 21 & 2 \\
     2 & 9 & 22 & 81 \\
     2 & 10 & 25 & 12\\
     2 & 12 & 28 & 465 \\
     2 & 13 & 29 & 217841 \\
     3 & 7 & 20 & 31918 \\
     \hline
\end{tabular}
\end{center}

For the enumerations in Theorem \ref{TheoremEnumeration}, we used the SAT Modulo Symmetries software package \cite{SMS}. By far the most difficult calculation was finding the 23 graphs of order 22 that contain no $B_5$ and no $\overline{B}_6$, which took 985401 seconds. We summarize the results in the following table.

  \begin{center}
        \begin{tabular}{|c|c|c|c|c|}
            \hline
            $r$ & $s$ & $R(B_r,B_s)$ &  \#critical Ramsey graphs & Computation time (seconds)\\
            \hline
            $1$ & $1$ & $6$ & $1$ & $1$\\ 
            $1$ & $2$ & $7$ & $4$ & $1$\\ 
            $1$ & $3$ & $9$ & $8$ & $1$\\ 
            $1$ & $4$ & $11$ & $7$ & $1$\\ 
            
            $1$ & $5$ & $13$ & $8$ & $1$\\ 
            $1$ & $6$ & $15$ & $8$ & $23$\\ 
$1$ & $7$ & $17$ & $10$ & $1367$\\ 
$1$ & $8$ & $19$ & $10$ & $125564$ \\
$2$ & $2$ & $10$ & $1$& $1$\\
$2$ & $3$ & $11$ & $4$ \cite{ShaoXuBoPan}& $1$\\
 $2$ & $4$ & $13$ & $6$ & $1$ \\ 
  $2$ & $5$ & $16$ & $1$ & $1$ \\
  $2$ & $6$ & $17$ & $3$ \cite{BlackLevenRadz} & $1$ \\
  $2$ & $7$ & $18$ & $65$ \cite{BlackLevenRadz} & $4$ \\
    $        2 $&$ 8 $&$ 21 $&$ 1 $ & $5$ \\
     $       2 $&$ 9 $&$ 22 $&$ 72 $ & $154$\\
      $      2 $&$ 10 $&$ 25 $&$ 5 $ & $22$\\
      $2$ & $11$ & $28$ & $1$ & $2$ \\
      $2$ & $12$ & $28$ & $10$ & $985$ \\
      $3$ & $3$ & $14$ & $1$ & $1$ \\ 
       $3 $&$ 4 $&$ 15 $&$ 1 $ \cite{ShaoXuBoPan} & $1$\\
      $3$ & $5$ & $17$ & $10$ & $2$ \\
       $     3 $&$ 6 $&$ 19 $&$ 4 $ & $60$\\
       $4$ & $4$ & $17$ & $1$ & $1$ \\
        $          4 $&$ 5 $&$ 19 $&$ 27 $ & $931$\\
       $5$ & $5$ & $21$ & $247$ & $4250$ \\
        $         5 $&$ 6 $&$ 23 $&$ 23 $ & $985401$\\
       $6$ & $6$ & $26$ & $15$ & $28264$ \\
            \hline 
        \end{tabular}
    \end{center}

\section{Acknowledgments}
The author would like to thank Sam Buss for many helpful discussions, comments, and computing assistance.

\bibliography{RamseyBooks}
\bibliographystyle{abbrv}

\appendix
\section{Graph Data} 
Here we give the graphs that give lower bounds for book Ramsey numbers. 

The following table gives block-circulant graph with 2 blocks, and recall from Section \ref{SectionBlockCirculant} that such a graph can be specified as $(D_{11},D_{12},D_{22})$ where $D_{ij}$ is the set of column indices (zero-indexed) with nonzero entries in the first row of the block $C_{ij}$. In all of our cases we are showing bounds of the form $R(B_r,B_s) \ge 2m+1$, and we have $D_{22} = \overline{D}_{11} \setminus \{0\}$, that is the complement of $D_{11}$ in $\Z_{m} \setminus \{0\}$. Therefore to write these graphs more compactly, we specify only $D_{11}$ and $D_{12}$. 

\scriptsize
\begin{center}
\setlength\extrarowheight{0pt}
\renewcommand{\arraystretch}{1}
\begin{tabular}{|l|l|l|}

\hline
    bound & $D_{11}$ & $D_{12}$ \\
    \hline
$R(B_8,B_8) \ge 33$ & $\{2,4,7,8,9,12,14 \}$ & $\{0,1,2,4,5,11,13,14 \}$ \\
\hline 
$R(B_{11},B_{11}) \ge 45$ & \makecell[l]{$\{5,6,7,8,9,11,13,14,15,$\\$16,17\}$} & \makecell[l]{$\{1,2,4,6,7,9,10,14,16,19,20\}$}\\
\hline 
$R(B_{14},B_{14}) \ge 57$ & \makecell[l]{$ \{1,3,4,5,8,13,14,15,20,$\\$23,24,25,27 \} $} & \makecell[l]{$\{ 2,4,6,9,10,16,17,18,19,$\\$21,22,23,24,27\}$} \\
\hline 
$R(B_{16},B_{16}) \ge 65$ & \makecell[l]{$ \{2,3,8,9,11,12,14,16,18,20,$\\$21,23,24,29,30\} $} & \makecell[l]{$\{ 6,9,11,13,15,16,19,22,23,24,$\\$26,27,28,29,30,31\}$} \\
\hline 
    \hline
    $R(B_7,B_8) \ge 31$ & $\{1,3,6,9,12,14\}$ & $\{0,3,4,7,8,9,10\}$ \\
    \hline $R(B_9,B_{10}) \ge 39$ & $\{4,5,6,7,8,11,12,13,14,15\}$ & $\{0,2,3,4,6,7,9,12,17\} $ \\
    \hline$R(B_{10},B_{11}) \ge 43$ & $\{2,3,4,5,7,14,16,17,18,19\}$ & $\{0,1,5,8,11,12,14,16,17,18\}$ \\
    \hline 
    $R(B_{11},B_{12}) \ge 47$ &$\{4,6,8,9,10,13,14,15,17,19 \}$ &$\{4,5,6,9,12,13,16,18,19,21,22\}$ \\
    \hline 
    $R(B_{13},B_{14}) \ge 55$ &$\{2, 4, 6, 9, 11, 12, 15, 16, 18, 21, 23, 25\}$ &$\{1, 3, 4, 5, 7, 8, 16, 17, 19, 20, 24, 25, 26 \} $ \\
    \hline 
    $R(B_{15},B_{16}) \ge 63$ & \makecell[l]{$\{1, 2, 3, 6, 8, 14, 15, 16, 17, 23,$\\ $25, 28, 29, 30\}$} & \makecell[l]{$\{ 1, 2, 3, 4, 5, 6, 8, 9, 10, 14,$\\ $16, 19, 22, 25, 29\}$} \\
    \hline 
    $R(B_{16},B_{17}) \ge 67$ & \makecell[l]{$\{ 1, 2, 8, 10, 12, 13, 14,$\\ $16, 17, 19, 20, 21, 23, 25, 31, 32\}$} & \makecell[l]{$\{0, 5, 6, 7, 10, 12, 13, 16, $\\$19, 20, 21, 23, 24, 26, 28, 29\}$} \\
    \hline 
    $R(B_{17},B_{18}) \ge 71$ & \makecell[l]{$\{1, 2, 3, 5, 7, 8, 12, 13,$\\ $16, 19, 22, 23, 27, 28, 30, 32, 33, 34\} $} & \makecell[l]{$\{0, 2, 6, 8, 11, 12, 13, $\\$15, 17, 18, 24, 25, 26, 27, 29, 33, 34\}$}\\
    \hline
     $R(B_{19},B_{20}) \ge 79$ & \makecell[l]{$\{1,4,7,8,9,12,13,14,16,23,$\\$25,26,27,30,31,32,35,38\} $} & \makecell[l]{$\{0,1,2,3,10,11,14,17,21,23,$\\$25,27,28,29,30,31,34,35,37\}$}\\
    \hline
    \hline
    $R(B_5,B_7) \ge 25$ & $\{ 2,4,5,7,8,10\} $ & $\{0,3,4,6,11\} $ \\
    \hline 
    $R(B_6,B_8) \ge 29$ & $\{1,4,6,7,8,10,13 \}$ & $\{ 0,4,5,6,8,9\}$ \\
    \hline 
    $R(B_7,B_9) \ge 33$ & $\{1,2,4,5,11,12,14,15\}$ & $\{3,5,6,8,10,14,15 \}$ \\
    \hline
    $R(B_8,B_{10}) \ge 37$ & $\{1,3,7,8,10,11,15,17 \}$ & $\{0,3,8,9,13,14,15,16\}$ \\
    \hline
    $R(B_9,B_{11}) \ge 41$ & $\{3,4,7,8,10,12,13,16,17 \} $ & $\{0,1,2,3,4,6,11,13,18 \}$ \\
    \hline
    $R(B_{10},B_{12})\ge 45$ & $\{ 1,5,7,8,10,11,12,14,15,17,21\}$ & $\{1,2,3,5,6,7,9,12,14,15 \}$ \\ 
    \hline 

    $R(B_{11},B_{13}) \ge 49$ & $\{1, 2, 4, 8, 9, 11, 13, 15, 16, 20, 22, 23 \}$ & $\{0, 5, 6, 10, 11, 12, 13, 14, 16, 19, 20 \}$ \\
    \hline 

    $R(B_{12},B_{14}) \ge 53$ & \makecell[l]{$\{3,7,8,9,11,12,13,14,15,17,$\\$18,19,23 \}$} & \makecell[l]{$\{0,2,4,7,9,13,16,17,18,19,$\\$20,25 \}$} \\
    \hline 

    $R(B_{13},B_{15}) \ge 57$ & \makecell[l]{$\{1,2,3,4,9,10,13,15,18,19,$\\$24,25,26,27 \}$} & \makecell[l]{$\{0,1,3,5,9,10,13,15,17,18,$\\$20,21,24 \}$} \\
    \hline 

    $R(B_{14},B_{16}) \ge 61$ & \makecell[l]{$\{1,2,3,4,9,12,14,16,18,21,$\\$26,27,28,29 \}$} & \makecell[l]{$\{0,1,2,5,6,8,9,12,19,20,$\\$22,26,27,28 \}$ }\\
    \hline 

    $R(B_{15},B_{17}) \ge 65$ & \makecell[l]{$\{1,4,5,6,7,9,10,16,22,$\\$23,25,26,27,28,31\}$} & \makecell[l]{$\{0,1,2,8,10,14,15,18,21,$\\$22,23,25,26,28,30\}$} \\
    \hline 

$R(B_{16},B_{18}) \ge 69$ & \makecell[l]{$\{1,3,4,5,8,9,15,16,18,19,$\\$25,26,29,30,31,33\}$} & \makecell[l]{$\{0,1,2,4,6,14,17,19,20,22,$\\$25,26,27,28,29,31\}$} \\
    \hline 
    
\end{tabular}
\end{center}

For our other bounds, the matrices we found are not block-circulant, so here we give their adjacency matrices. 

\begin{itemize} 
\item 
$R(B_2,B_{8}) \ge 21$:

  [0, 0, 0, 0, 1, 0, 1, 0, 1, 1, 1, 1, 0, 0, 1, 1, 0, 0, 0, 0]

  [0, 0, 0, 1, 0, 1, 1, 0, 0, 0, 1, 1, 1, 0, 1, 0, 0, 1, 0, 0]

  [0, 0, 0, 1, 0, 0, 1, 1, 1, 0, 0, 0, 1, 0, 1, 1, 0, 0, 0, 0]

  [0, 1, 1, 0, 1, 0, 0, 0, 1, 1, 0, 1, 0, 1, 0, 0, 0, 0, 1, 0]

  [1, 0, 0, 1, 0, 0, 1, 1, 0, 0, 0, 0, 1, 0, 0, 0, 0, 1, 1, 1]

  [0, 1, 0, 0, 0, 0, 1, 1, 1, 1, 0, 0, 0, 0, 0, 1, 0, 0, 1, 1]

  [1, 1, 1, 0, 1, 1, 0, 0, 0, 0, 0, 0, 0, 1, 0, 0, 1, 0, 0, 0]

  [0, 0, 1, 0, 1, 1, 0, 0, 0, 1, 1, 1, 0, 0, 0, 0, 0, 1, 0, 0]

  [1, 0, 1, 1, 0, 1, 0, 0, 0, 0, 1, 0, 0, 0, 0, 0, 1, 1, 0, 1]

  [1, 0, 0, 1, 0, 1, 0, 1, 0, 0, 0, 0, 1, 1, 1, 0, 1, 0, 0, 0]

  [1, 1, 0, 0, 0, 0, 0, 1, 1, 0, 0, 0, 1, 1, 0, 0, 0, 0, 1, 0]

  [1, 1, 0, 1, 0, 0, 0, 1, 0, 0, 0, 0, 0, 0, 0, 1, 1, 0, 0, 1]

  [0, 1, 1, 0, 1, 0, 0, 0, 0, 1, 1, 0, 0, 0, 0, 1, 1, 0, 0, 1]

  [0, 0, 0, 1, 0, 0, 1, 0, 0, 1, 1, 0, 0, 0, 0, 1, 0, 1, 0, 1]

  [1, 1, 1, 0, 0, 0, 0, 0, 0, 1, 0, 0, 0, 0, 0, 0, 0, 1, 1, 1]

  [1, 0, 1, 0, 0, 1, 0, 0, 0, 0, 0, 1, 1, 1, 0, 0, 0, 1, 1, 0]

  [0, 0, 0, 0, 0, 0, 1, 0, 1, 1, 0, 1, 1, 0, 0, 0, 0, 1, 1, 0]

  [0, 1, 0, 0, 1, 0, 0, 1, 1, 0, 0, 0, 0, 1, 1, 1, 1, 0, 0, 0]

  [0, 0, 0, 1, 1, 1, 0, 0, 0, 0, 1, 0, 0, 0, 1, 1, 1, 0, 0, 0]

  [0, 0, 0, 0, 1, 1, 0, 0, 1, 0, 0, 1, 1, 1, 1, 0, 0, 0, 0, 0]

\item 
$R(B_2,B_9) \ge 22$:

[0, 0, 0, 0, 1, 1, 0, 0, 0, 0, 1, 0, 1, 0, 0, 1, 1, 0, 0, 1, 1]

[0, 0, 0, 0, 0, 1, 1, 0, 0, 0, 1, 1, 1, 1, 0, 1, 0, 1, 0, 0, 0]

[0, 0, 0, 1, 0, 1, 0, 0, 1, 1, 1, 0, 1, 0, 0, 0, 0, 0, 0, 0, 0]

[0, 0, 1, 0, 0, 0, 1, 1, 0, 0, 0, 0, 1, 0, 0, 1, 0, 1, 0, 1, 1]

[1, 0, 0, 0, 0, 0, 1, 1, 1, 0, 0, 1, 0, 0, 0, 1, 0, 1, 1, 0, 0]

[1, 1, 1, 0, 0, 0, 0, 1, 1, 0, 0, 0, 0, 0, 1, 0, 0, 1, 0, 1, 0]

[0, 1, 0, 1, 1, 0, 0, 0, 1, 0, 1, 0, 0, 0, 1, 0, 1, 0, 0, 1, 0]

[0, 0, 0, 1, 1, 1, 0, 0, 0, 0, 1, 1, 0, 1, 1, 0, 1, 0, 0, 0, 0]

[0, 0, 1, 0, 1, 1, 1, 0, 0, 0, 0, 0, 0, 1, 0, 0, 0, 0, 0, 0, 1]

[0, 0, 1, 0, 0, 0, 0, 0, 0, 0, 1, 1, 0, 0, 1, 1, 1, 1, 0, 1, 0]

[1, 1, 1, 0, 0, 0, 1, 1, 0, 1, 0, 0, 0, 0, 0, 0, 0, 0, 1, 0, 1]

[0, 1, 0, 0, 1, 0, 0, 1, 0, 1, 0, 0, 1, 0, 0, 0, 0, 0, 0, 1, 1]

[1, 1, 1, 1, 0, 0, 0, 0, 0, 0, 0, 1, 0, 0, 1, 0, 1, 0, 1, 0, 0]

[0, 1, 0, 0, 0, 0, 0, 1, 1, 0, 0, 0, 0, 0, 0, 1, 1, 0, 1, 1, 1]

[0, 0, 0, 0, 0, 1, 1, 1, 0, 1, 0, 0, 1, 0, 0, 1, 0, 0, 1, 0, 1]

[1, 1, 0, 1, 1, 0, 0, 0, 0, 1, 0, 0, 0, 1, 1, 0, 0, 0, 0, 0, 0]

[1, 0, 0, 0, 0, 0, 1, 1, 0, 1, 0, 0, 1, 1, 0, 0, 0, 1, 0, 0, 0]

[0, 1, 0, 1, 1, 1, 0, 0, 0, 1, 0, 0, 0, 0, 0, 0, 1, 0, 1, 0, 1]

[0, 0, 0, 0, 1, 0, 0, 0, 0, 0, 1, 0, 1, 1, 1, 0, 0, 1, 0, 1, 0]

[1, 0, 0, 1, 0, 1, 1, 0, 0, 1, 0, 1, 0, 1, 0, 0, 0, 0, 1, 0, 0]

[1, 0, 0, 1, 0, 0, 0, 0, 1, 0, 1, 1, 0, 1, 1, 0, 0, 1, 0, 0, 0]

\item $R(B_2,B_{10}) \ge 25$

[0, 0, 0, 0, 0, 0, 0, 0, 0, 0, 0, 0, 0, 0, 0, 1, 1, 1, 1, 1, 1, 1, 1, 1]

[0, 0, 0, 0, 0, 0, 0, 0, 0, 0, 0, 1, 1, 1, 1, 0, 0, 0, 0, 1, 1, 1, 1, 1]

[0, 0, 0, 0, 0, 0, 0, 0, 0, 1, 1, 0, 0, 1, 1, 0, 0, 1, 1, 0, 0, 1, 1, 1]

[0, 0, 0, 0, 0, 0, 0, 1, 1, 0, 1, 0, 0, 0, 1, 0, 1, 0, 1, 0, 1, 0, 1, 1]

[0, 0, 0, 0, 0, 0, 1, 0, 1, 0, 0, 0, 1, 0, 1, 1, 0, 0, 1, 0, 1, 1, 0, 1]

[0, 0, 0, 0, 0, 0, 1, 1, 0, 1, 0, 1, 0, 0, 0, 0, 0, 0, 1, 0, 1, 1, 1, 0]

[0, 0, 0, 0, 1, 1, 0, 0, 0, 0, 1, 0, 0, 1, 0, 0, 1, 1, 0, 1, 0, 0, 1, 1]

[0, 0, 0, 1, 0, 1, 0, 0, 0, 0, 0, 0, 1, 1, 0, 1, 0, 1, 0, 1, 0, 1, 0, 1]

[0, 0, 0, 1, 1, 0, 0, 0, 0, 1, 0, 1, 0, 1, 0, 0, 0, 1, 0, 1, 0, 1, 1, 0]

[0, 0, 1, 0, 0, 1, 0, 0, 1, 0, 0, 0, 1, 0, 0, 0, 1, 0, 0, 1, 1, 0, 0, 1]

[0, 0, 1, 1, 0, 0, 1, 0, 0, 0, 0, 1, 1, 0, 0, 1, 0, 0, 0, 1, 1, 1, 0, 0]

[0, 1, 0, 0, 0, 1, 0, 0, 1, 0, 1, 0, 0, 0, 0, 1, 0, 1, 1, 0, 0, 0, 0, 1]

[0, 1, 0, 0, 1, 0, 0, 1, 0, 1, 1, 0, 0, 0, 0, 0, 1, 1, 1, 0, 0, 0, 1, 0]

[0, 1, 1, 0, 0, 0, 1, 1, 1, 0, 0, 0, 0, 0, 0, 1, 1, 0, 1, 0, 1, 0, 0, 0]

[0, 1, 1, 1, 1, 0, 0, 0, 0, 0, 0, 0, 0, 0, 0, 1, 1, 1, 0, 1, 0, 0, 0, 0]

[1, 0, 0, 0, 1, 0, 0, 1, 0, 0, 1, 1, 0, 1, 1, 0, 0, 0, 0, 0, 0, 0, 1, 0]

[1, 0, 0, 1, 0, 0, 1, 0, 0, 1, 0, 0, 1, 1, 1, 0, 0, 0, 0, 0, 0, 1, 0, 0]

[1, 0, 1, 0, 0, 0, 1, 1, 1, 0, 0, 1, 1, 0, 1, 0, 0, 0, 0, 0, 1, 0, 0, 0]

[1, 0, 1, 1, 1, 1, 0, 0, 0, 0, 0, 1, 1, 1, 0, 0, 0, 0, 0, 1, 0, 0, 0, 0]

[1, 1, 0, 0, 0, 0, 1, 1, 1, 1, 1, 0, 0, 0, 1, 0, 0, 0, 1, 0, 0, 0, 0, 0]

[1, 1, 0, 1, 1, 1, 0, 0, 0, 1, 1, 0, 0, 1, 0, 0, 0, 1, 0, 0, 0, 0, 0, 0]

[1, 1, 1, 0, 1, 1, 0, 1, 1, 0, 1, 0, 0, 0, 0, 0, 1, 0, 0, 0, 0, 0, 0, 0]

[1, 1, 1, 1, 0, 1, 1, 0, 1, 0, 0, 0, 1, 0, 0, 1, 0, 0, 0, 0, 0, 0, 0, 0]

[1, 1, 1, 1, 1, 0, 1, 1, 0, 1, 0, 1, 0, 0, 0, 0, 0, 0, 0, 0, 0, 0, 0, 0]

\end{itemize}


\end{document}